\documentclass[12pt]{amsart}
\usepackage{amssymb}
\usepackage{multicol}
\usepackage{colordvi}
\usepackage{graphicx}

\usepackage{mathrsfs}

\usepackage{vmargin}
\setpapersize{USletter}
\setmargrb{2.3cm}{2.3cm}{2.3cm}{2.3cm} % --- sets all four margins.  Cool.

\newtheorem{theorem}{Theorem} %[section]
\newtheorem{proposition}[theorem]{Proposition}
\newtheorem{lemma}[theorem]{Lemma}
\newtheorem{corollary}[theorem]{Corollary}

\theoremstyle{definition}

\newtheorem{example}[theorem]{Example}
\newtheorem{remark}[theorem]{Remark}

\newcounter{FNC}[page]
\def\fauxfootnote#1{{\addtocounter{FNC}{2}\Magenta{$^\fnsymbol{FNC}$}%
     \let\thefootnote\relax\footnotetext{\Magenta{$^\fnsymbol{FNC}$#1}}}}
\newcommand{\spec}{\mathop{\rm spec}\nolimits}

\newcommand{\pls}{\mathscr{P}^{\infty}}
\newcommand{\scrA}{{\mathscr{A}}}
\newcommand{\coscrA}{{\mathit{co}\mathscr{A}}}
\newcommand{\scrL}{{\mathscr{L}}}

\newcommand{\calA}{{\mathcal{A}}}
\newcommand{\calI}{{\mathcal{I}}}
\newcommand{\calO}{{\mathcal{O}}}
\newcommand{\calV}{{\mathcal{V}}}
\newcommand{\calX}{{\mathcal{X}}}

\newcommand{\bm}{{\bf m}}
\renewcommand{\P}{{\mathbb P}}
\newcommand{\C}{{\mathbb C}}
\newcommand{\bS}{{\mathbb S}}
\newcommand{\N}{{\mathbb N}}
\newcommand{\R}{{\mathbb R}}
\newcommand{\T}{{\mathbb T}}
\newcommand{\U}{{\mathbb U}}
\newcommand{\Z}{{\mathbb Z}}

\newcommand{\ini}{\mbox{\rm in}}

\newcommand{\Log}{\mathop{\rm Log}\nolimits}

\newcommand{\Hom}{\mathop{\rm Hom}\nolimits}

\newcommand{\DeCo}[1]{\Blue{#1}}
\newcommand{\demph}[1]{\DeCo{{\sl #1}}}

\newcommand{\hooklongrightarrow}{\lhook\joinrel\longrightarrow}
\newcommand{\twoheadlongrightarrow}{\relbar\joinrel\twoheadrightarrow}
\title{The Phase Limit Set of a Variety}
\author{Mounir Nisse}
\address{Mounir Nisse\\
          Department of Mathematics\\
         Texas A\&M University\\
         College Station\\
         Texas \ 77843\\
         USA}
\email{nisse@math.tamu.edu}
\urladdr{www.math.tamu.edu/\~{}nisse}
\author{Frank Sottile}
\address{Frank Sottile\\
         Department of Mathematics\\
         Texas A\&M University\\
         College Station\\
         Texas \ 77843\\
         USA}
\email{sottile@math.tamu.edu}
\urladdr{www.math.tamu.edu/\~{}sottile}
\thanks{Research of Sottile supported in part by NSF grant DMS-1001615 and the Institut
  Mittag-Leffler.} 
\begin{document}

\begin{abstract}
 A coamoeba is the image of a subvariety of a complex torus under the argument map
 to the real torus.
 We describe the structure of the boundary of the coamoeba of a variety,
 which we relate to its logarithmic limit set.
 Detailed examples of lines in three-dimensional space illustrate and motivate these results.
\end{abstract}

\maketitle
%%%%%%%%%%%%%%%%%%%%%%%%%%%%%%%%%%%%%%%%%%%%%%%%%%%%%%%%%%%%%%%%%%%%%%%%%%%%%

%%%%%%%%%%%%%%%%%%%%%%%%%%%%%%%%%%%%%%%%%%%%%%%%%%%%%%%%%%%%%%%%%%%%

\section{Introduction}

A coamoeba is the image of a subvariety of a complex torus under the argument map
to the real torus.
Coamoebae are cousins to amoebae, which are images of subvarieties under
the coordinatewise logarithm map $z\mapsto\log|z|$.
Amoebae were introduced by Gelfand, Kapranov, and Zelevinsky in 1994~\cite{GKZ}, and have
subsequently been widely studied~\cite{KeOk,Mi,PaRu,Pur}.
Coamoebae were introduced by Passare in a talk in 2004, 
and they appear to have many beautiful and interesting
properties. 
For example, coamoebae of $\calA$-discriminants in dimension two are unions of
two non-convex polyhedra~\cite{NiPa}, and a hypersurface coamoeba has an associated
arrangement of codimension one tori contained in its closure~\cite{Nisse}. 

Bergman~\cite{Berg} introduced the logarithmic limit set $\scrL^\infty(X)$ of a
subvariety $X$ of the torus as the set of limiting directions of points in its amoeba.
Bieri and Groves~\cite{BiGr}  showed that $\scrL^\infty(X)$
is a rational polyhedral complex in the sphere.
Logarithmic limit sets are now called tropical algebraic varieties~\cite{SS}.
For a hypersurface  $\calV(f)$, logarithmic limit set $\scrL^\infty(\calV(f))$ consists of the
directions of non-maximal cones in the outer normal fan of the Newton polytope of $f$.
We introduce a similar object for coamoebae 
and establish a structure theorem for coamoebae 
similar to that of Bergman and of Bieri and Groves for amoebae.

Let \demph{$\coscrA(X)$} be the coamoeba of a subvariety $X$ of $(\C^*)^n$ with ideal $I$. 
The  \demph{phase limit set} of $X$, \DeCo{$\pls(X)$}, is the set of accumulation
points of arguments of sequences in $X$ with unbounded logarithm.
For $w\in\R^n$, the initial variety $\ini_w X\subset(\C^*)^n$ is the variety of the initial
ideal of $I$. 
The fundamental theorem of tropical geometry asserts that $\ini_w X\neq \emptyset$
exactly when the direction of $-w$ lies in $\scrL^\infty(X)$.
We establish its analog for coamoebae.

%%%%%%%%%%%%%%%%%%%%%%%%%%%%%%%%%%%%%%%%%%%%%%%%%%%%%%%%%%%%%%%%%%
\begin{theorem}\label{T:one}
  The closure of $\coscrA$ is $\coscrA(X)\cup \pls(X)$, and 
\[
   \pls(X)\ =\  \bigcup_{w\neq 0} \coscrA(\ini_w X)\,.
\]
\end{theorem}
%%%%%%%%%%%%%%%%%%%%%%%%%%%%%%%%%%%%%%%%%%%%%%%%%%%%%%%%%%%%%%%%%%

Johansson~\cite{Johansson} used different methods to prove this when $X$ is a
complete intersection. 

The cone over the logarithmic limit set admits the structure of a rational polyhedral fan
$\Sigma$ in which all weights $w$ in the relative interior of a cone $\sigma\in\Sigma$
give the same initial scheme $\ini_w X$.
Thus the union in Theorem~\ref{T:one} is finite and is indexed by the images of these cones
$\sigma$ in the logarithmic limit set of $X$.
The logarithmic limit set or tropical algebraic variety is a combinatorial shadow of $X$
encoding many properties of $X$.
While the coamoeba of $X$ is typically not purely combinatorial (see the examples of lines in
$(\C^*)^3$ in Section~\ref{S:lines}), the phase limit set does provide a combinatorial skeleton
which we believe will be useful in the further study of coamoebae.

We give definitions and background in Section~\ref{S:defs}, and 
detailed examples of lines in three-dimensional space in Section~\ref{S:lines}.
These examples are reminiscent of the concrete descriptions of amoebae of lines in~\cite{Th}.
We prove Theorem~\ref{T:one} in Section~\ref{S:phase}.

%%%%%%%%%%%%%%%%%%%%%%%%%%%%%%%%%%%%%%%%%%%%%%%%%%%%%%%%%%%%%%%%%%%%

%%%%%%%%%%%%%%%%%%%%%%%%%%%%%%%%%%%%%%%%%%%%%%%%%%%%%%%%%%%%%%%%%%%%%%%%%%%%%
\section{Coamoebae, tropical varieties and initial ideals}\label{S:defs}

As a real algebraic group, the set $\T:=\C^*$ of invertible complex numbers is
isomorphic to $\R\times\U$ under the map $(r,\theta)\mapsto e^{r+\sqrt{-1}\theta}$.
Here,  $\U$ is the set of complex numbers of norm 1
which may be identified with $\R/2\pi\Z$.
The inverse map is $z\mapsto (\log|z|,\arg(z))$.

Let $M$ be a free abelian group of finite rank and $N=\Hom(M,\Z)$ its dual group.
We use $\langle\cdot,\cdot\rangle$ for the pairing between $M$ and $N$.
The group ring $\C[M]$ is the ring of Laurent polynomials with exponents in
$M$.
It is the coordinate ring of a torus $\T_N$ which is identified with 
$N\otimes_\Z\T = \Hom(M,\T)$, the set of group homomorphisms $M\to \T$.
There are natural maps $\Log\colon\T_N\to\R_N=N\otimes_\Z\R$ and 
$\arg\colon\T_N\to\U_N=N\otimes_\Z\U$, which are induced by the maps
$\C^*\ni z\mapsto \log|z|$ and $z\mapsto\arg(z)\in\U$.
Maps $N\to N'$ of free abelian groups induce corresponding maps $\T_N\to\T_{N'}$ of tori,
and also of $\R_N$ and $\U_N$.
If $n$ is the rank of $N$, we may identify $N$ with $\Z^n$, which identifies $\T_N$ with 
$\T^n$, $\U_N$ with $\U^n$, and $\R_N$ with $\R^n$.

The \demph{amoeba $\scrA(X)$} of a subvariety $X\subset\T_N$ is its image 
under the map $\Log\colon\T_N\to\R_N$, and the 
\demph{coamoeba} \DeCo{$\coscrA(X)$} of $X$ is the image of $X$ under the argument
map $\arg\colon\T_N\to\U_N$. 
An amoeba has a geometric-combinatorial structure at infinity encoded by the 
logarithmic limit set~\cite{Berg,BiGr}.
Coamoebae similarly have phase limit sets which have a related combinatorial
structure that we define and study in Section~\ref{S:phase}.

If we identify $\C^*$ with $\R^2\setminus\{(0,0)\}$, then the map 
$\arg\colon\C^*\to\U$ given by $(a,b)\mapsto (a,b)/\sqrt{a^2+b^2}$ is a real
algebraic map.
Thus, coamoebae, as they are the image of a real algebraic subset of the real algebraic
variety $\T_N$ under the real algebraic map $\arg$, are semialgebraic subsets of
$\U_N$~\cite{BPR}. 
It would be very interesting to study them as semi-algebraic sets, in
particular, what are the equations and inequalities satisfied by a coamoeba?
When $X$ is a Grassmannian, such a description would generalize
Richter-Gebert's five-point condition for phirotopes from rank two to arbitrary
rank~\cite{BKRG}. 

Similarly, we may replace the map $\C^*\ni z\mapsto \log|z|\in\R$ in the definition of
amoebae by the map $\C^*\ni z\mapsto |z|\in\R^+:=\{r\in\R\mid r>0\}$ to obtain the 
\demph{algebraic amoeba} of $X$, which is a subset of $\R^+_N$.
The algebraic amoeba is a semi-algebraic subset of $\R^+_N$, and we also 
ask for its description as a semi-algebraic set.

%%%%%%%%%%%%%%%%%%%%%%%%%%%%%%%%%%%%%%%%%%%%%%%%%%%%%%%%%%%%%%%%%%%%%%%%%%%%%%%%%%%%
\begin{example}\label{Ex:lineP2}
 Let $\ell\subset \T^2$ be defined by $x+y+1=0$.
 The coamoeba $\coscrA(\ell)$ is the set of points of $\U^2$ of the form
 $(\arg(x),\pi+\arg(x+1))$ for $x\in\C\setminus\{0,-1\}$. 
 If $x$ is real, then these points are $(\pm\pi,0)$, $(\pm\pi,\pm\pi)$, and
 $(0,\pm\pi)$ if $x$ lies in the intervals $(-\infty,-1)$, $(-1,0)$, and $(0,\infty)$
 respectively. 
 For other values, consider the picture below in the complex plane.
 \[
   \begin{picture}(180,53)(-40,-1)
    \put(0,0){\includegraphics{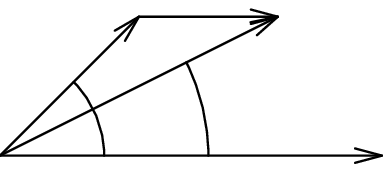}}
    \put(18,43){$x$}
    \put(82,43){$x+1$}
    \put(115,2){$\R$}    \put(-9,-2){$0$}
    \put(84,22){$\arg(x+1)$} \put(82,24.5){\vector(-4,-1){20}}
    \put(-40,16){$\arg(x)$}\put(-4,18.5){\vector(4,-1){30}}
   \end{picture}
 \]
 For $\arg(x)\not\in\{0,\pi\}$ fixed, $\pi+\arg(x{+}1)$ can take on any value strictly
 between $\pi+\arg(x)$ (for $w$ near $\infty$) and
 $0$  (for $x$ near $0$), and thus 
 $\coscrA(\ell)$ consists of the three points $(\pi,0)$, $(\pi,\pi)$, and
 $(0,\pi)$ and the interiors of the two triangles displayed below in the
 fundamental domain $[-\pi,\pi]^2\subset \R^2$ of $\U^2$.
 This should be understood modulo $2\pi$, so that $\pi=-\pi$.
  \begin{equation}\label{Eq:two_triangles}
   \raisebox{-50pt}{%
    \begin{picture}(215,96)(-20,-12)
     \put(-2,-2){\includegraphics{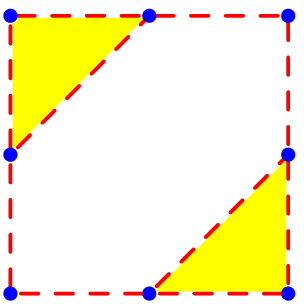}}
     \put(-25,-2){$-\pi$} \put(-15,38){$0$} \put(-15,77){$\pi$}
     \put(-10,-12){$-\pi$} \put(35,-12){$0$} \put(77,-12){$\pi$}
     \put(100,2){\vector(1,0){70}} \put(115,8){$\arg(x)$}
     \put(93,15){\vector(0,1){56}} \put(97,42){$\arg(y)=\pi+\arg(x{+}1)$}
   \end{picture}}
  \end{equation}
 The coamoeba is the complement of the region 
 \[
    \{(\alpha,\beta)\in[-\pi,\pi]^2\;:\;
   |\alpha-\beta|\ \leq\ \pi=\arg(-1)\}\,,
 \]
together with the three images of real points $(\pm\pi,0)$, $(\pm\pi,\pm\pi)$, and
 $(0,\pm\pi)$.

Given a general line $ax+by+c=0$ with $a,b,c\in\C^*$, we may 
replace $x$ by $cx'/a$ and $y$ by $cy'/b$, to obtain the line $x'+y'+1=0$, 
with coamoeba~\eqref{Eq:two_triangles}.
This transformation rotates the coamoeba~\eqref{Eq:two_triangles} 
by $\arg(a/c)$ horizontally and $\arg(b/c)$ vertically.
\end{example}
%%%%%%%%%%%%%%%%%%%%%%%%%%%%%%%%%%%%%%%%%%%%%%%%%%%%%%%%%%%%%%%%%%%%%%%%%

Let $f\in\C[M]$ be a polynomial with support $\calA\subset M$,
 \begin{equation}\label{Eq:laurentPolynomial}
   f\ =\ \sum_{\bm\in\calA} c_\bm\cdot \xi^\bm\,,\qquad c_\bm\in\C^*\,,
 \end{equation}
where we write $\xi^\bm$ for the element of $\C[M]$ corresponding to $\bm\in M$.
Given $w\in\R_N$, let $w(f)$ be the minimum of $\langle\bm,w\rangle$ for $\bm\in\calA$.
Then the initial form $\ini_w f$ of $f$
with respect to $w\in\R_N$ is the polynomial $\ini_wf\in\C[M]$ defined by
\[
   \DeCo{\ini_wf}\ :=\ \sum_{\langle\bm,w\rangle=w(f)} c_\bm\cdot \xi^\bm\,.
\]
Given an ideal $I\subset\C[M]$ and $w\in\R_N$, the \demph{initial ideal} with
respect to $w$ is 
\[
   \DeCo{\ini_w I}\ :=\ \langle\ini_w f\mid f\in  I\rangle\ \subset\ 
    \C[M]\,.
\]
Lastly, when $I$ is the ideal of a subvariety $X$, the \demph{initial scheme} 
$\DeCo{\ini_wX}\subset\T_N$ is defined by the initial ideal $\ini_w I$.

The sphere $\DeCo{\bS_N}:=(\R_N\setminus\{0\})/\R^+$ is the set of directions in $\R_N$.
Write $\pi\colon \R_N\setminus\{0\}\to\bS_N$ for the projection.
The \demph{logarithmic limit set $\scrL^\infty(X)$} of a subvariety $X$ of $\T_N$
is the set of accumulation points in $\bS_N$ of sequences $\{\pi(\Log(x_n))\}$ where
$\{x_n\}\subset X$ is an unbounded set.
A sequence $\{x_n\}\subset\T_N$ is unbounded if its sequence of
logarithms $\{\Log(x_n)\}$ is unbounded.

A \demph{rational polyhedral cone} $\sigma\subset\R_N$ is the set of points $w\in\R_N$
which satisfy finitely many inequalities and equations of the form
\[
   \langle \bm,w \rangle\ \geq\ 0
  \qquad\mbox{and}\qquad
   \langle \bm',w \rangle\ =\ 0\,,
\]
where $\bm,\bm'\in M$.
The \demph{dimension} of $\sigma$ is the dimension of its linear span, and \demph{faces}
of $\sigma$ are proper subsets of $\sigma$ obtained by replacing some inequalities by
equations. 
The relative interior of $\sigma$ consists of its points not lying in any face.
Also, $\sigma$ is determined by $\sigma\cap N$, which is a finitely generated subsemigroup 
of $N$.

A \demph{rational polyhedral fan $\Sigma$} is a collection of rational polyhedral cones in
$\R_N$ in which every two cones of $\Sigma$ meet along a common face.

%%%%%%%%%%%%%%%%%%%%%%%%%%%%%%%%%%%%%%%%%%%%%%%%%%%%%%%%%%%%%%%%%%%%
\begin{theorem}\label{T:FTTG}
  The cone in $\R_N$ over the negative $-\scrL^\infty(X)$ of the logarithmic limit set of
  $X$ is the set of $w\in\R_N$ such that $\ini_wX\neq\emptyset$.
  Equivalently, it is the set of $w\in\R_N$ such that for every $f\in\C[M]$ lying in the
  ideal $I$ of $X$, $\ini_w f$ is not a monomial.

  This cone over $-\scrL^\infty(X)$ admits the structure of a rational polyhedral fan
  $\Sigma$ with the property that  if $u,w$ lie in the relative interior of a cone $\sigma$
  of $\Sigma$, then $\ini_u I=\ini_w I$.
\end{theorem}
%%%%%%%%%%%%%%%%%%%%%%%%%%%%%%%%%%%%%%%%%%%%%%%%%%%%%%%%%%%%%%%%%%%%

It is important to take $-\scrL^\infty(X)$.
This is correct as we use the tropical convention of minimum, which is forced by our use
of toric varieties to prove Theorem~\ref{T:one} in Section~\ref{S:tropicalCompact}.

We write $\ini_\sigma I$ for the initial ideal defined by points in the relative interior
of a cone $\sigma$ of $\Sigma$.
The fan structure $\Sigma$ is not canonical, for it depends upon an identification
$M\xrightarrow{\;\sim\;}\Z^n$.
Moreover, it may be the case that $\sigma\neq\tau$, but $\ini_\sigma I=\ini_\tau I$.

Bergman~\cite{Berg} defined the logarithmic limit set of a subvariety
of the torus $\T_N$, and 
Bieri and Groves~\cite{BiGr} showed it was a finite union of convex polyhedral cones.
The connection to initial ideals was made more explicit through work of Kapranov~\cite{EKL}
and the above form is adapted from Speyer and Sturmfels~\cite{SS}.
The logarithmic limit set of $X$ is now called the tropical algebraic variety of $X$,
and this latter work led to the field of tropical geometry.

%%%%%%%%%%%%%%%%%%%%%%%%%%%%%%%%%%%%%%%%%%%%%%%%%%%%%%%%%%%%%%%%%%%%

%%%%%%%%%%%%%%%%%%%%%%%%%%%%%%%%%%%%%%%%%%%%%%%%%%%%%%%%%%%%%%%%%%%%%%%%%%%%%
\section{Lines in space}\label{S:lines}
%%%%%%%%%%%%%%%%%%%%%%%%%%%%%%%%%%%%%%%%%%%%%%%%%%%%%%%%%%%%%%%%%

We consider coamoebae of lines in three-dimensional space.
We will work in the torus $\T\P^3$ of $\P^3$, which is the quotient of $\T^4$ by the diagonal
torus $\Delta_\T$ and similarly in $\U\P^3$, the quotient of $\U^4$ by the diagonal
$\Delta_\U:=\{(\theta,\theta,\theta,\theta)\mid\theta\in \U\}$.
By \demph{coordinate lines and planes} in $\U\P^3$ we mean the images in $\U\P^3$ of lines
and planes in $\U^4$ parallel to some coordinate plane.

Let $\ell$ be a line in $\P^3$ not lying in any coordinate plane.
Then $\ell$ has a parameterization
 \begin{equation}\label{Eq:parametrization}
  \phi\ \colon\ \P^1\ni[s:t]\ \longmapsto\ 
   [\ell_0(s,t):\ell_1(s,t):\ell_2(s,t):\ell_3(s,t)]\,,
 \end{equation}
where $\ell_0,\ell_1,\ell_2,\ell_3$ are non-zero linear forms which do not all vanish at
the same point.
For $i=0,\dotsc,3$, let $\zeta_i\in\P^1$ be the zero of $\ell_i$.
The configuration of these zeroes determine the coamoeba of $\ell\cap\T\P^3$, which we
will simply write as $\coscrA(\ell)$.

Suppose that two zeroes coincide, say $\zeta_3=\zeta_2$.
Then $\ell_3=a\ell_2$ for some $a\in\C^*$, and so $\ell$ lies in the translated
subtorus $z_3=a z_2$ and its coamoeba $\coscrA(\ell)$ lies in the coordinate subspace of
$\U_3$ defined by $\theta_3=\arg(a)+\theta_2$.
In fact, $\coscrA(\ell)$ is pulled back from the coamoeba of the projection of $\ell$ to the
$\theta_3=0$ plane. 
It follows that if there are only two distinct roots among $\zeta_0,\dotsc,\zeta_3$, then
$\coscrA(\ell)$ is a coordinate line of $\U_3$.
If three of the roots are distinct, then (up to a translation) the projection of the coamoeba
$\coscrA(\ell)$ to the  $\theta_3=0$ plane looks like~\eqref{Eq:two_triangles} so that
$\coscrA(\ell)$ consists of two triangles lying in a coordinate plane. 

For each $i=0,\dotsc,3$ define a function depending upon a point $[s\colon t]\in\P^1$ and
$\theta\in\U$ by
\[
   \varphi_i(s,t;\theta)\ =\ \left\{\begin{array}{lcl}
     \theta&\ &\mbox{if $\ell_i(s,t)=0$}\,,\\
     \arg(\ell_i(s,t))&\ &\mbox{otherwise}\end{array}\right..
\]
For each $i=0,\dotsc,3$, let $h_i$ be the image in $\U\P^3$ of $\U$
under the map 
\[
    \theta\ \longmapsto\ 
    [\varphi_0(\zeta_i,\theta),\,\varphi_1(\zeta_i,\theta),\,
     \varphi_2(\zeta_i,\theta),\,\varphi_3(\zeta_i,\theta)]\,.
\]

%%%%%%%%%%%%%%%%%%%%%%%%%%%%%%%%%%%%%%%%%%%%%%%%%%%%%%%%%%%%%%%%%
\begin{lemma}
  For each $i=0,\dotsc,3$, $h_i$ is a coordinate line in $\U\P^3$ that consists of
  accumulation points of $\coscrA(\ell)$.
\end{lemma}
%%%%%%%%%%%%%%%%%%%%%%%%%%%%%%%%%%%%%%%%%%%%%%%%%%%%%%%%%%%%%%%%%

This follows from Theorem~\ref{T:one}.
%
%  Proved later
%
For the main idea, note that $\arg\circ\phi(\zeta_i+\epsilon e^{\theta\sqrt{-1}})$ for
$\theta\in\U$ is a curve in $\U\P^3$ whose Hausdorff 
distance to the line $h_i$ approaches 0 as $\epsilon\to 0$.
The \demph{phase limit set} of $\ell$ is the union of these four lines.

%%%%%%%%%%%%%%%%%%%%%%%%%%%%%%%%%%%%%%%%%%%%%%%%%%%%%%%%%%%%%%%%%
\begin{lemma}\label{L:constant}
  Suppose that the zeroes $\zeta_0,\zeta_1,\zeta_2$ are distinct.
  Then 
 \[
    \P^1\setminus\{\zeta_0,\zeta_1,\zeta_2\}\ \ni\ x
    \ \longmapsto\ 
      \arg(\ell_0(x),\ell_1(x),\ell_2(x))\ \in\ \U^3/\Delta_\U\ =\ \U\P^2
\]
  is constant along each arc
  of the circle in $\P^1$ through $\zeta_0,\zeta_1,\zeta_2$.
\end{lemma}
%%%%%%%%%%%%%%%%%%%%%%%%%%%%%%%%%%%%%%%%%%%%%%%%%%%%%%%%%%%%%%%%%

%%%%%%%%%%%%%%%%%%%%%%%%%%%%%%%%%%%%%%%%%%%%%%%%%%%%%%%%%%%%%%%%%
\begin{proof}
  After changing coordinates in $\P^1$ and translating in $\U\P^2$ (rotating coordinates), we
  may assume that these roots are $\infty, 0, -1$ and so the circle becomes the real line.
  Choosing affine coordinates, we may assume that $\ell_0=1$, $\ell_1=x$ and $\ell_2=x+1$, so
  that we are in the situation of Example~\ref{Ex:lineP2}.
  Then the statement of the lemma is the computation there for $x$ real in
  which we obtained the coordinate points $(\pi,0)$, $(\pi,\pi)$, and $(0,\pi)$.
\end{proof}
%%%%%%%%%%%%%%%%%%%%%%%%%%%%%%%%%%%%%%%%%%%%%%%%%%%%%%%%%%%%%%%%%

%%%%%%%%%%%%%%%%%%%%%%%%%%%%%%%%%%%%%%%%%%%%%%%%%%%%%%%%%%%%%%%%%
\begin{lemma}\label{L:disjoint}
 The phase limit lines $h_0$, $h_1$, $h_2$, and $h_3$ are disjoint if and only if the roots
 $\zeta_0,\dotsc,\zeta_3$ do not all lie on a circle.
\end{lemma}
%%%%%%%%%%%%%%%%%%%%%%%%%%%%%%%%%%%%%%%%%%%%%%%%%%%%%%%%%%%%%%%%%

%%%%%%%%%%%%%%%%%%%%%%%%%%%%%%%%%%%%%%%%%%%%%%%%%%%%%%%%%%%%%%%%%
\begin{proof}
 Suppose that two of the limit lines meet, say $h_0$ and $h_1$.
 Without loss of generality, we suppose that we have chosen coordinates on $\U^4$ and
 $\P^1$ so that $\zeta_i\in\C$ and $\ell_i(x)=x-\zeta_i$ for $i=0,\dotsc,3$.
 Then there are points $\alpha,\beta,\theta\in\U$ such that 
 \begin{multline*}
  \qquad (\varphi_0(\zeta_0,\alpha),\,\varphi_1(\zeta_0,\alpha),\,
   \varphi_2(\zeta_0,\alpha),\,\varphi_3(\zeta_0,\alpha))\\
   =\ 
  (\varphi_0(\zeta_1,\beta),\,\varphi_1(\zeta_1,\beta),\,
   \varphi_2(\zeta_1,\beta),\,\varphi_3(\zeta_1,\beta))\ 
   +\ (\theta,\theta,\theta,\theta)\,.\qquad 
 \end{multline*}
 Comparing the last two components, we obtain
\[
   \arg(\zeta_0-\zeta_2)\ =\  \arg(\zeta_1-\zeta_2)\ +\ \theta
   \qquad\mbox{and}\qquad
   \arg(\zeta_0-\zeta_3)\ =\  \arg(\zeta_1-\zeta_3)\ +\ \theta\,,
\]
 and so the zeroes $\zeta_0,\dotsc,\zeta_3$ have the configuration below.
\[
  \begin{picture}(113,68)(-12,0)
    \put(0,1){\includegraphics[height=60pt]{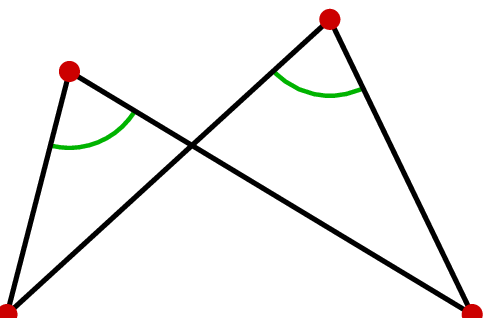}}
    \put(14,24){$\theta$}    \put(56,33.5){$\theta$}
    \put(1,52){$\zeta_3$}    \put(66,61){$\zeta_2$}
    \put(-12,0){$\zeta_0$}    \put(94,0){$\zeta_1$}
  \end{picture}
\]
 But then $\zeta_0,\dotsc,\zeta_3$ are cocircular.
 Conversely, if $\zeta_0,\dotsc,\zeta_3$ lie on a circle $C$, then by Lemma~\ref{L:constant}
 the lines $h_i$ and $h_j$ meet only if $\zeta_i$ and $\zeta_j$ are the endpoints of an arc of 
 $C\setminus\{\zeta_0,\dotsc,\zeta_3\}$. 
\end{proof}
%%%%%%%%%%%%%%%%%%%%%%%%%%%%%%%%%%%%%%%%%%%%%%%%%%%%%%%%%%%%%%%%%

%%%%%%%%%%%%%%%%%%%%%%%%%%%%%%%%%%%%%%%%%%%%%%%%%%%%%%%%%%%%%%%%%
\begin{lemma}\label{L:immersion}
 If the roots $\zeta_0,\dotsc,\zeta_3$ do not all lie on a circle, then the map
\[
   \arg\circ \phi\ \colon\ \P^1\setminus\{\zeta_0,\zeta_1,\zeta_2,\zeta_3\}
    \ \longrightarrow\ \U\P^3
\]
 is an immersion.
\end{lemma}
%%%%%%%%%%%%%%%%%%%%%%%%%%%%%%%%%%%%%%%%%%%%%%%%%%%%%%%%%%%%%%%%%

%%%%%%%%%%%%%%%%%%%%%%%%%%%%%%%%%%%%%%%%%%%%%%%%%%%%%%%%%%%%%%%%%
\begin{proof}
 Let $x\in\P^1\setminus\{\zeta_0,\zeta_1,\zeta_2,\zeta_3\}$, which we consider to be a
 real two-dimensional manifold.
 After possibly reordering the roots, the circle $C_1$ containing
 $x,\zeta_0,\zeta_1$ meets the circle $C_2$ containing $x,\zeta_2,\zeta_3$ transversally at
 $x$.
%
%   If two circles are mutually tangent at $x$, apply inversion to make one a line.
% Then the other is a circle tangent to the line.   There are two remaining ways to 
% match the roots in pairs, if one matching gives circles that are mutually tangent 
% at $x$, then the other cannot, for one of the mutually tangent circles contains 
% the other
%
 Under the derivative of the map $\arg\circ\phi$, tangent vectors at $x$ to $C_1$ and $C_2$
 are taken to nonzero vectors $(0,0,u_1,v_1)$ and $(u_2,v_2,0,0)$ in the tangent space to
 $\U^4$. 
 Furthermore, as the four roots do not all lie on a circle, we cannot have both
 $u_1=v_1$ and $u_2=v_2$, and so this derivative has full rank two at $x$, as a map from 
 $\P^1\setminus\{\zeta_0,\zeta_1,\zeta_2,\zeta_3\}\to \U\P^3$, which proves the lemma.
\end{proof}
%%%%%%%%%%%%%%%%%%%%%%%%%%%%%%%%%%%%%%%%%%%%%%%%%%%%%%%%%%%%%%%%%

By these lemmas, there is a fundamental difference between the coamoebae of lines
when the roots of the lines $\ell_i$ are cocircular and when they are not.
We examine each case in detail.
First, choose coordinates so that $\zeta_0=\infty$.
After dehomogenizing and separately rescaling each affine coordinate (e.g.~identifying
$\U\P^3$ 
with $\U^3$ and applying phase shifts to each coordinate $\theta_1,\theta_2,\theta_3$ of
$\U^3$), we may assume that the map~\eqref{Eq:parametrization} parametrizing $\ell$ 
is 
 \begin{equation}\label{Eq:Nparametrization}
  \phi\ \colon\  \C\ni x\ \longmapsto (x-\zeta_1,\,x-\zeta_2,\,x-\zeta_3)\ \in\ \C^3\,.
 \end{equation}

Suppose first that the four roots are cocircular.
As $z_0=\infty$, the other three lie on a real line in $\C$, which we may assume is $\R$.
That is, if the four roots are cocircular, then up to coordinate change, we may assume that
the line $\ell$ is real and the affine parametrization~\eqref{Eq:Nparametrization} is also
real. 
For this reason, we will call such lines $\ell$ \demph{real lines}.
We first study the boundary of $\coscrA(\ell)$.
Suppose that $x$ lies on a contour $C$ in the upper half plane as in Figure~\ref{F:contour}
%%%%%%%%%%%%%%%%%%%%%%%%%%%%%%%%%%%%%%%%%%%%%%%%%%%%%%%%%%%%%%%%%%%%%%%%%%%%%%%%%%%%%%%
\begin{figure}[htb]
  \begin{picture}(223,65)
    \put(0,10){\includegraphics{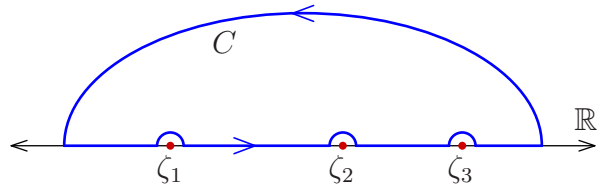}}
    \put(55,0){$\zeta_1$} \put(120,0){$\zeta_2$} \put(165,0){$\zeta_3$}
    \put(213,18){$\R$} \put(76,47){$C$}
  \end{picture}
 \caption{Contour in upper half-plane}
 \label{F:contour}
\end{figure}
%%%%%%%%%%%%%%%%%%%%%%%%%%%%%%%%%%%%%%%%%%%%%%%%%%%%%%%%%%%%%%%%%%%%%%%%%%%%%%%%%%%%%%%%
that contains semicircles of
radius $\epsilon$ centered at each root and a semicircle of radius $1/\epsilon$ centered at 0,
but otherwise lies along the real axis, for $\epsilon$ a sufficiently small positive number. 
Then $\arg(\phi(w))\in\U^3$ is constant on the four segments of $C$ lying along
$\R$ with respective values 
 \begin{equation}\label{Eq:values}
   (\pi,\pi,\pi)\,,\quad
   (0,\pi,\pi)\,,\quad
   (0,0,\pi)\,,\quad\mbox{and}\quad
   (0,0,0)\,,
 \end{equation}
moving from left to right.
On the semicircles around $\zeta_1$, $\zeta_2$, and $\zeta_3$, two of the coordinates are
essentially constant (but not quite equal to either 0 or $\pi$!), while the third
decreases from $\pi$ to 0. 
Finally, on the large semicircle, the three coordinates are nearly equal and increase from 
$(0,0,0)$ to $(\pi,\pi,\pi)$.
The image $\arg(\phi(C))$ can be made as close as we please to the quadrilateral in
$\U^3$ connecting the points of~\eqref{Eq:values} in cyclic order, when $\epsilon$ is
sufficiently small. 
Thus the image of the upper half plane under the map $\arg\circ\,\phi$ is a
relatively open membrane in $\U^3$ that spans the quadrilateral.
It lies within the convex hull of this quadrilateral, which is computed
using the affine structure induced from $\R^3$ by the quotient
$\U^3=\R^3/(2\pi\Z)^3$. 

For this, observe that its projection in any of the four coordinate directions parallel
to its edges is one of the triangles of the coamoeba of the projected line in
$\C\P^2$ of Example~\ref{Ex:lineP2}, and the convex hull of the quadrilateral is the
intersection of the four preimages of these triangles.

Because $\ell$ is real, the image of the lower half plane is
isomorphic to the image of the upper half plane, under the map
$(\theta_1,\theta_2,\theta_3)\mapsto(-\theta_1,-\theta_2,-\theta_3)$ and so 
the coamoeba is symmetric in the origin of $\U^3$ and consists of two quadrilateral patches
that meet at their vertices.
Here are two views of the coamoeba of the line where the roots are
$\infty,-1/2, 0, 3/2$:
\[
  \includegraphics[height=1.9in]{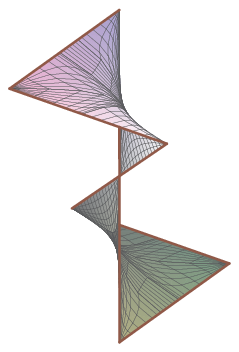}
   \qquad
  \includegraphics[height=1.9in]{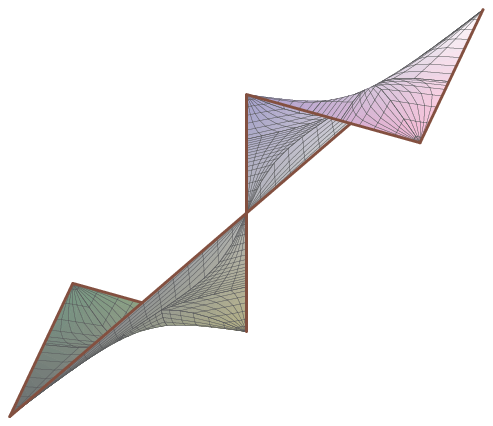}
\]

Now suppose that the roots $\zeta_0,\dotsc,\zeta_3$ do not all lie on a circle.
By Lemma~\ref{L:disjoint}, the four phase limit lines $h_1,\dotsc,h_3$ are disjoint
and the map from $\ell$ to the coamoeba is an immersion.
Figure~\ref{F:symmetric} shows two views of 
the coamoeba in a fundamental domain of $\U\P^3$ when the roots are 
$\infty, 1,\zeta,\zeta^2$, where $\zeta$ is a primitive third root of infinity.
%%%%%%%%%%%%%%%%%%%%%%%%%%%%%%%%%%%%%%%%%%%%%%%%%%%%%%%%%%%%%%%%%
\begin{figure}[htb]
  \includegraphics[height=1.9in]{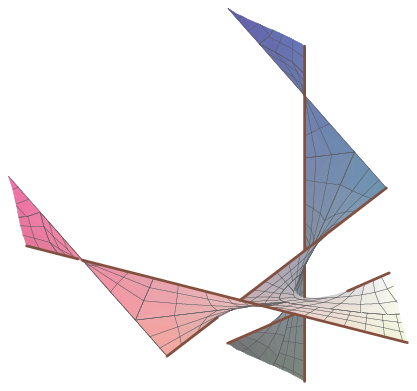}
   \qquad
  \includegraphics[height=1.9in]{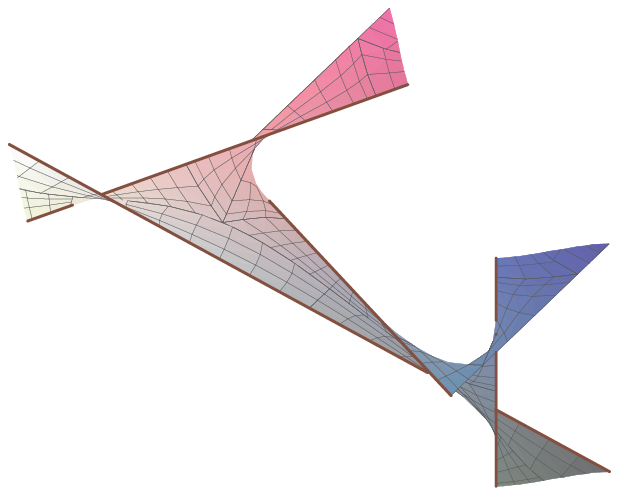}
\caption{Two views of the coamoeba of a symmetric line.}
\label{F:symmetric}
\end{figure}
%%%%%%%%%%%%%%%%%%%%%%%%%%%%%%%%%%%%%%%%%%%%%%%%%%%%%%%%%%%%%%%%%
This and other pictures of coamoebae of lines are animated on the webpage~\cite{WWW-coAmoeba}.

The projection of this coamoeba along a coordinate direction
(parallel to one of the phase limit lines $h_i$) gives a coamoeba of a line in $\T\P^2$,
as we saw in Example~\ref{Ex:lineP2}.
The line $h_i$ is mapped to the interior of one triangle and the vertices of the triangles are
the images of line segments lying on the coamoeba.
These three line segments come from the three arcs of the circle through the three roots other
than $\zeta_i$, the root corresponding to $h_i$.

\begin{proposition}\label{P:linesegments}
 The interior of the coamoeba of a general line in $\T\P^3$ contains $12$ line segments in
 triples parallel to each of the four coordinate directions. 
\end{proposition}

The symmetric coamoeba we show in Figure~\ref{F:symmetric} has six additional line segments,
two each coming from the three longitudinal circles through a third root of unity and
$0,\infty$. 
Two such segments are visible as pinch points in the leftmost view in
Figure~\ref{F:symmetric}. 
We ask:  What is the maximal number of line segments on a coamoeba of a line in $\T\P^3$?

%%%%%%%%%%%%%%%%%%%%%%%%%%%%%%%%%%%%%%%%%%%%%%%%%%%%%%%%%%%%%%%%%%%%
\section{Structure of the Phase limit set}\label{S:phase}
%%%%%%%%%%%%%%%%%%%%%%%%%%%%%%%%%%%%%%%%%%%%%%%%%%%%%%%%%%%%%%%%%%%%

The \demph{phase limit set $\pls(X)$} of a complex subvariety $X\subset \T_N$ is the set
of all accumulation points of sequences $\{\arg(x_n)\mid n\in \N\}\subset \U_N$, where 
$\{ x_n\mid n\in \N\}\subset X$ is an unbounded sequence.
For $w\in N$, $\ini_w X\subset \T_N$ is the (possibly empty) initial scheme of $X$, whose
ideal is the initial ideal $\ini_w I$, where $I$ is the ideal of $X$.
Our main result is that the phase limit set of $X$ is the union of the coamoebae of all
its initial schemes. 
\medskip

%%%%%%%%%%%%%%%%%%%%%%%%%%%%%%%%%%%%%%%%%%%%%%%%%%%%%%%%%%%%%%%%%%%%%
\noindent{\bf Theorem~\ref{T:one}.}
  The closure of $\coscrA$ is $\coscrA(X)\cup \pls(X)$, and 
\[
   \pls(X)\ =\  \bigcup_{w\neq 0} \coscrA(\ini_w X)\,.
\]
\medskip
%%%%%%%%%%%%%%%%%%%%%%%%%%%%%%%%%%%%%%%%%%%%%%%%%%%%%%%%%%%%%%%%%%%%%

%%%%%%%%%%%%%%%%%%%%%%%%%%%%%%%%%%%%%%%%%%%%%%%%%%%%%%%%%%%%%%%%%%%%%
\begin{remark}
 This is a finite union.
 By Theorem~\ref{T:FTTG}, $\ini_w X $ is non-empty
 only when $w$ lies in the cone over the logarithmic set $\scrL^\infty(X)$, which 
 can be given the structure of a finite union of rational
 polyhedral cones such that any two points in the relative interior of the same
 cone $\sigma$ have the same initial scheme. 
 If we write $\ini_\sigma X $ for the initial scheme corresponding to a cone $\sigma$,
 then the torus $\T_{\langle\sigma\rangle}\simeq(\C^*)^{\dim \sigma}$ acts on 
 $\ini_\sigma X$ by translation (see, e.g.\ Corollary~\ref{C:iniGeometry}). 
 (Here, $\langle\sigma\rangle\subset N$ is the span of 
 $\sigma\cap N$, a free abelian group of rank $\dim \sigma$.)
 This implies that $\coscrA(\ini_\sigma X)$ is a union of orbits of
 $\coscrA(\T_{\langle\sigma\rangle})=\U_{\langle\sigma\rangle}$, and thus that 
 $\dim (\coscrA(\ini_\sigma X))\leq 2\dim(X)-\dim(\sigma)$.
\end{remark}
%%%%%%%%%%%%%%%%%%%%%%%%%%%%%%%%%%%%%%%%%%%%%%%%%%%%%%%%%%%%%%%%%%%

This discussion implies the following proposition.

%%%%%%%%%%%%%%%%%%%%%%%%%%%%%%%%%%%%%%%%%%%%%%%%%%%%%%%%%%%%%%%%%%%
\begin{proposition}
 Let $X\subset\T_N$ be a subvariety and suppose that $\T_X\subset\T_N$ is the largest
 subtorus acting on $X$.
 Then we have $\dim \coscrA(X) \leq   \min\{ \dim \T_N, 2\dim X-\dim \T_X\}$.
\end{proposition}
%%%%%%%%%%%%%%%%%%%%%%%%%%%%%%%%%%%%%%%%%%%%%%%%%%%%%%%%%%%%%%%%%%%

We prove Theorem~\ref{T:one} in the next two subsections.

%%%%%%%%%%%%%%%%%%%%%%%%%%%%%%%%%%%%%%%%%%%%%%%%%%%%%%%%%%%%%%%%%%%
\subsection{Coamoebae of initial schemes}\label{S:initial}

We review the standard dictionary relating initial ideals to toric
degenerations in the context of subvarieties of $\T_N$~\cite[Ch.~6]{GKZ}. 
Let $X\subset\T_N$ be a subvariety with ideal $I\subset\C[M]$.
We study $\ini_w I$ and the initial schemes 
$\ini_w X=\calV(\ini_w I)\subset\T_N$ for $w\in N$.
Since $\ini_0 I=I$, so that $\ini_0X=X$, we may assume that $w\neq 0$.
As $N$ is the lattice of one-parameter subgroups of $\T_N$, $w$ corresponds to a
one-parameter subgroup written as $\C^*\ni t\mapsto t^w\in\T_N$.
Define $\calX\subset\C\times \T_N$ by
 \begin{equation}\label{E:toric_degeneration}
   \calX\ :=\ \{(t,x)\in \C^*\times\T_N\mid t^w\cdot x\in X\}\,.
 \end{equation}
The fiber of $\calX$ over a point $t\in\C^*$ is $t^{-w}X$.
Let \demph{$\overline{\calX}$} be the closure of $\calX$ in $\C\times\T_N$, and set
\demph{$X_0$} to be the fiber of $\overline{\calX}$ over $0\in\C$.

%%%%%%%%%%%%%%%%%%%%%%%%%%%%%%%%%%%%%%%%%%%%%%%%%%%%%%%%%%%%%%%%%%%%%%%%%%%%%%%%%%%%%%%%%%%%
\begin{proposition}\label{P:initial_scheme}
   $X_0\ =\ \ini_w X$.
\end{proposition}
%%%%%%%%%%%%%%%%%%%%%%%%%%%%%%%%%%%%%%%%%%%%%%%%%%%%%%%%%%%%%%%%%%%%%%%%%%%%%%%%%%%%%%%%%%%%

%%%%%%%%%%%%%%%%%%%%%%%%%%%%%%%%%%%%%%%%%%%%%%%%%%%%%%%%%%%%%%%%%%%%%%%%%%%%%%%%%%%%%%%%%%%%
\begin{proof}
 We first describe the ideal $\calI$ of $\calX$.
 For $\bm\in M$, the element $\xi^\bm\in\C[M]$ takes the value
 $t^{\langle \bm, w\rangle}\in\C^*$ on the element $t^w\in\T_N$, and so 
 if $x\in\T_N$, then $\xi^\bm$ takes the value 
 $t^{\langle \bm, w\rangle}\xi^\bm(x)=t^{\langle \bm, w\rangle}\bm(x)$ on $t^w x$.
 Given a polynomial $f\in\C[M]$ of the form
\[
      f\ :=\ \sum_{\bm\in\calA} c_\bm \xi^\bm\,,\qquad (c_\bm\in\C^*)
\]
  define the polynomial $f(t)\in\C[t,t^{-1}][M]$ by 
 \begin{equation}\label{Eq:deformed_poly}
   f(t)\ :=\ \sum_{\bm\in\calA} c_\bm t^{\langle \bm,w\rangle} \xi^\bm\,.
 \end{equation}
 Then $f(t)(x)=f(t^w x)$, so $\calI$ is generated by the polynomials
 $f(t)$~\eqref{Eq:deformed_poly}, for $f\in I$.
 A general element of $\calI$ is a linear combination of translates 
 $t^af(t)$ of such polynomials, for $a\in\Z$.

 If we set $w(f)$ to be the minimal exponent of $t$ occurring in $f(t)$, then 
\[
   \ini_w f\ =\ \sum_{\langle \bm,w\rangle=w(f)} c_\bm \xi^\bm\,,
\]
 and 
\[
   t^{-w(f)}f(t)\ =\ \ini_w f\ +\ 
     \sum_{\langle\bm,w\rangle>w(f)} t^{\langle\bm,w\rangle-w(f)}c_\bm \xi^\bm\,.
\]
 This shows that $\calI\cap\C[t][M]$ is generated by polynomials
 $t^{-w(f)}f(t)$, where $f\in I$.
 Since $\ini_w f\in\C[M]$ and the remaining terms are divisible by $t$, we see that the
 ideal of $X_0$ is generated by $\{\ini_w f\mid f\in \calI\}$, which 
 completes the proof.
\end{proof}
%%%%%%%%%%%%%%%%%%%%%%%%%%%%%%%%%%%%%%%%%%%%%%%%%%%%%%%%%%%%%%%%%%%%%%%%%%%%%%%%%%%%%%%%%%%%

We use Proposition~\ref{P:initial_scheme} to prove one inclusion of Theorem~\ref{T:one},
that
 \begin{equation}\label{E:supset}
    \pls(X)\ \supset\ \bigcup_{w\in N\setminus\{0\}} \coscrA(\ini_w X)\,.
 \end{equation}
Fix $0\neq w\in N$, and let $\calX$, $\overline{\calX}$, and $X_0=\ini_w X$ be as in
Proposition~\ref{P:initial_scheme}, and let $x_0\in X_0$.
We show that $\arg(x_0)\in\pls(X)$.
Since $(0,x_0)\in\overline{\calX}$, there is an irreducible curve $C\subset\calX$ with
$(0,x_0)\in\overline{C}$.
The projection of $C\subset \C^*\times\T_N$ to $\C^*$ is dominant, so there
exists a sequence 
$\{(t_n,x_n)\mid n\in\N\}\subset C$ that converges to $(0,x_0)$ with each $t_n$ real and
positive.
Then $\arg(x_0)$ is the limit of the sequence $\{\arg(x_n)\}$.

For each $n\in\N$, set $\DeCo{y_n}:=t_n^w\cdot x_n\in X$.
Since $t_n$ is positive and real, every component of $t_n^w$ is positive and real, and
so  $\arg(y_n) =  \arg(x_n)$.
Thus $\arg(x_0)$ is the limit of the sequence $\{\arg(y_n)\}$.
Since $x_n$ converges to $x_0$ and $t_n$ converges to $0$, the sequence
$\{y_n\}\subset X$ is unbounded, which implies that $\arg(x_0)$ lies in the 
phase limit set of $X$.
This proves~\eqref{E:supset}.

%%%%%%%%%%%%%%%%%%%%%%%%%%%%%%%%%%%%%%%%%%%%%%%%%%%%%%%%%%%%%%%%%%%
\subsection{Coamoebae and tropical compactifications}\label{S:tropicalCompact}
%%%%%%%%%%%%%%%%%%%%%%%%%%%%%%%%%%%%%%%%%%%%%%%%%%%%%%%%%%%%%%%%%%%%
We complete the proof of Theorem~\ref{T:one} by establishing the other inclusion,
\[
  \pls(X)\ \subset\ \bigcup_{w\in N\setminus\{0\}} \coscrA(\ini_w X)\,.
\]
Suppose that $\{x_n\mid n\in\N\}\subset X$ is an unbounded sequence.
To study the accumulation points of the sequence $\{\arg(x_n)\mid n\in\N\}$, we use a
compactification of $X$ that is adapted to its inclusion in $\T_N$.
Suitable compactifications are Tevelev's tropical
compactifications~\cite{Tevelev}, for in these the boundary of $X$ is composed of initial
schemes $\ini_w X$ of $X$ in a manner we describe below.

%%%%%%%%%%%%%%%%%%%%%%%%%%%%%%%%%%%%%%%%%%%%%%%%%%%%%%%%%%%%%%%%%%%
By Theorem~\ref{T:FTTG}, the cone over the logarithmic limit set $\scrL^\infty(X)$
of $X$ is the support of a rational polyhedral fan $\Sigma$ whose cones
$\sigma$ have the property that all initial ideals $\ini_w I$ coincide for $w$ in the
relative interior of $\sigma$.

Recall the construction of the toric variety $Y_\Sigma$ associated to a fan
$\Sigma$~\cite{Fulton},~\cite[Ch.~6]{GKZ}.
For $\sigma\in\Sigma$, set
 \begin{eqnarray*}
   \DeCo{\sigma^\vee} &:=&
      \{\bm\in M\mid \langle\bm,w\rangle\geq 0\mbox{ for all }w\in\sigma\}\,,
        \mbox{ and}\\
   \DeCo{\sigma^\perp}&:=& 
     \{\bm\in M\mid \langle\bm,w\rangle=    0\mbox{ for all }w\in\sigma\}\,.
 \end{eqnarray*}
Set $\DeCo{V_\sigma}:=\spec\C[\sigma^\vee]$ and 
$\DeCo{\calO_\sigma}:=\spec\C[\sigma^\perp]$, which is
naturally isomorphic to $\T_N/\T_{\langle\sigma\rangle}$, where 
$\DeCo{\langle\sigma\rangle}\subset N$
is the subgroup generated by $\sigma\cap N$.
The map $\bm\mapsto\bm\otimes\bm$ determines a comodule map
$\C[\sigma^\vee]\to\C[\sigma^\vee]\otimes\C[M]$, which induces the action of the
torus $\T_N$ on $V_\sigma$.
Its orbits correspond to faces of the cone $\sigma$ with the smallest orbit $\calO_\sigma$
corresponding to $\sigma$ itself.
The inclusion $\sigma^\perp\subset\sigma^\vee$ is split by the semigroup map
 \begin{equation}\label{Eq:semigroupMap}
   \sigma^\vee\ni\bm\ \longmapsto\ 
    \left\{\begin{array}{rcl}
        \bm&\ &\mbox{if }\bm\in\sigma^\perp\\
          0&\ &\mbox{if }\bm\not\in\sigma^\perp
         \end{array}\right.\,,
 \end{equation}
which induces a map $\C[M]\twoheadrightarrow\C[\sigma^\perp]$, 
and thus we have the $\T_N$-equivariant split inclusion
 \begin{equation}\label{E:splitInclusion}
   \calO_\sigma\ \hooklongrightarrow\ V_\sigma\ 
   \stackrel{\pi_\sigma}{\relbar\joinrel\twoheadlongrightarrow}\ \calO_\sigma\,.
 \end{equation}
On orbits $\calO_\tau$ in $V_\sigma$, the map $\pi_\sigma$ is simply the quotient by
$\T_{\langle\sigma\rangle}$.

If $\sigma,\tau\in\Sigma$ with $\sigma\subset\tau$, then $\sigma^\vee\supset\tau^\vee$,
so $\C[\sigma^\vee]\supset\C[\tau^\vee]$, and so $V_\sigma\subset V_\tau$.
Since the quotient fields of $\C[\sigma^\vee]$ and $\C[M]$ coincide, these
are inclusions of open sets, and these varieties $V_\sigma$ for $\sigma\in\Sigma$ glue
together along these natural inclusions to give the toric variety $Y_\Sigma$.
The torus $\T_N$ acts on $Y_\Sigma$ with an orbit $\calO_\sigma$ for each cone $\sigma$ of
$\Sigma$. 

Since $V_0=\T_N$, $Y_\Sigma$ contains $\T_N$ as a dense subset, and thus $X$ is a
(non-closed) subvariety.
Let $\overline{X}$ be the closure of $X$ in $Y_\Sigma$.
As the fan $\Sigma$ is supported on the cone over $\scrL^\infty(X)$, $\overline{X}$ will be
a tropical compactification of $X$ and $\overline{X}$ is
complete~\cite[Prop.~2.3]{Tevelev}. 
To understand the points of $\overline{X}\setminus X$, we study the intersection
$\overline{X}\cap V_\sigma$, which is defined by $I\cap\C[\sigma^\vee]$, as well as 
the intersection $\overline{X}\cap\calO_\sigma$, which is defined in $\C[\sigma^\perp]$ by
the image $I(\sigma)$ of $I\cap\C[\sigma^\vee]$ under the map 
$\C[\sigma^\vee]\twoheadrightarrow\C[\sigma^\perp]$ induced by~\eqref{E:splitInclusion}. 

%%%%%%%%%%%%%%%%%%%%%%%%%%%%%%%%%%%%%%%%%%%%%%%%%%%%%%%%%%%%%%%%%%%%
\begin{lemma}\label{L:initialIdeals}
  The initial ideal $\ini_\sigma I\subset\C[M]$ of $I$ is generated by $I(\sigma)$ under
  the inclusion $\C[\sigma^\perp]\hookrightarrow\C[M]$.
\end{lemma}
%%%%%%%%%%%%%%%%%%%%%%%%%%%%%%%%%%%%%%%%%%%%%%%%%%%%%%%%%%%%%%%%%%%%

%%%%%%%%%%%%%%%%%%%%%%%%%%%%%%%%%%%%%%%%%%%%%%%%%%%%%%%%%%%%%%%%%%%%
\begin{proof}
 Let $f\in I$.
 Since $\sigma$ is a cone in $\Sigma$, we have that $\ini_\sigma f=\ini_wf$
 for all $w$ in the relative interior of $\sigma$.
 Thus for $w\in\sigma$, the function $\bm\mapsto\langle\bm,w\rangle$ on exponents of
 monomials of $f$ is minimized on (a superset of) the support of $\ini_\sigma f$, and if
 $w$ lies in the relative interior of $\sigma$, then the minimizing set is the support of 
 $\ini_\sigma f$.
 Multiplying $f$ if necessary by $\xi^{-\bm}$, where $\bm$ is some monomial of $\ini_\sigma
 f$, we may assume that for every $w\in\sigma$, the linear function
 $\bm\mapsto\langle\bm,w\rangle$ is nonnegative on the support of $f$, so that
 $f\in\C[\sigma^\vee]$, and the function is zero on the support of $\ini_\sigma f$.
 Furthermore, if $w$ lies in the relative interior of $\sigma$, then it vanishes exactly
 on the support of $\ini_\sigma f$.
 Thus $\ini_\sigma f\in\C[\sigma^\perp]$, which completes the proof.
\end{proof}
%%%%%%%%%%%%%%%%%%%%%%%%%%%%%%%%%%%%%%%%%%%%%%%%%%%%%%%%%%%%%%%%%%%%

Since $\calO_\sigma=\T_N/\T_{\langle\sigma\rangle}$, Lemma~\ref{L:initialIdeals} has the
following geometric counterpart.

%%%%%%%%%%%%%%%%%%%%%%%%%%%%%%%%%%%%%%%%%%%%%%%%%%%%%%%%%%%%%%%%%%%%
\begin{corollary}\label{C:iniGeometry}
 $\T_{\langle\sigma\rangle}$ acts (freely) on $\ini_\sigma X$ by translation with 
 $\ini_\sigma X/\T_{\langle\sigma\rangle}=\overline{X}\cap\calO_\sigma$.
\end{corollary}
%%%%%%%%%%%%%%%%%%%%%%%%%%%%%%%%%%%%%%%%%%%%%%%%%%%%%%%%%%%%%%%%%%%%

%%%%%%%%%%%%%%%%%%%%%%%%%%%%%%%%%%%%%%%%%%%%%%%%%%%%%%%%%%%%%%%%%%%%
\begin{proof}[Proof of Theorem~$\ref{T:one}$]
 Let $\theta\in\pls(X)$ be a point in the phase limit set of $X$.
 Then there exists an unbounded sequence $\{x_n\mid n\in\N\}\subset X$ with
\[
   \lim_{n\to\infty} \arg(x_n)\ =\ \theta\,.
\]
 Since $\overline{X}$ is compact, the sequence $\{x_n\mid n\in\N\}$ has an accumulation
 point $x$ in $\overline{X}$.
 As the sequence is unbounded, $x\not\in\calO_0$, and so $x\in\overline{X}\setminus X$.
 Thus $x$ is a point of $\overline{X}\cap\calO_\sigma$ for some  cone $\sigma\neq 0$ of
 $\Sigma$.
 Replacing $\{x_n\}$ by a subsequence, we may assume that $\lim_n x_n=x$.

 Because the map $\pi_\sigma$~\eqref{E:splitInclusion} is continuous and is the identity on
 $\calO_\sigma$, we have that $\{\pi_\sigma(x_n)\}$ converges to $\pi_\sigma(x)=x$, and thus
 \begin{equation}\label{Eq:limits}
   \pi_\sigma(\theta)\ =\ \pi_\sigma \bigl(\lim_{n\to\infty}\arg(x_n)\bigr)\ =\ 
    \arg\bigl(\lim_{n\to\infty}\pi_\sigma (x_n)\bigr)\ =\ 
   \arg(x)\ \in\ \coscrA(\overline{X}\cap\calO_\sigma)\,.
 \end{equation}
 Corollary~\ref{C:iniGeometry} implies that $\coscrA(\overline{X}\cap\calO_\sigma)=
 \coscrA(\ini_\sigma X)/\U_\sigma$, as $\U_{\langle\sigma\rangle}=\arg(\T_{\langle\sigma\rangle})$.
 Recall that on $\calO_0$, $\pi_\sigma$ is the quotient by $\T_{\langle\sigma\rangle}$.
 Thus we conclude from~\eqref{Eq:limits} that $\theta\in\coscrA(\ini_\sigma X)$
 which completes the proof of Theorem~\ref{T:one} as $\ini_\sigma X =\ini_w X$ for any $w$
 in the relative interior of $\sigma$.
\end{proof}
%%%%%%%%%%%%%%%%%%%%%%%%%%%%%%%%%%%%%%%%%%%%%%%%%%%%%%%%%%%%%%%%%%%%

%%%%%%%%%%%%%%%%%%%%%%%%%%%%%%%%%%%%%%%%%%%%%%%%%%%%%%%%%%%%%%%%%%%%
\begin{example}
 In~\cite{Nisse}, the closure of a hypersurface coamoeba $\coscrA(\calV(f))$ for
 $f\in\C[M]$ was shown to contain a finite
 collection of \demph{codual hyperplanes}.
 These are translates of codimension one subtori
 $\U_\sigma$ for $\sigma$ a cone in the normal fan of the Newton polytope of $f$
 corresponding to an edge.
 By Theorem~\ref{T:one}, these translated tori are that part of the phase limit set of $X$
 corresponding to the cones $\sigma$ dual to the edges, specifically $\coscrA(\ini_\sigma X)$.
 Since $\sigma$ has dimension $n{-}1$, the torus $\T_\sigma$ acts with finitely
 many orbits on $\ini_\sigma X$, which is therefore a union of finitely many translates
 of $\T_\sigma$. 
 Thus $\coscrA(\ini_\sigma X)$ is a union of finitely many translates of
 $\U_\sigma$.

 The logarithmic limit set $\scrL^\infty(C)$ of a curve $C\subset\T_N$ is a finite
 collection of points in $\bS_N$.
 Each point gives a ray in the cone over $\scrL^\infty(C)$, and the components of $\pls(C)$
 corresponding to a ray $\sigma$ are finitely many translations of the dimension one
 subtorus $\U_\sigma$ of $\U_N$, which we referred to as lines in Section~\ref{S:lines}.
 These were the lines lying in the boundaries of the coamoebae $\coscrA(\ell)$ of the lines
 $\ell$ in $\T^2$ and $\T^3$.
\end{example}
%%%%%%%%%%%%%%%%%%%%%%%%%%%%%%%%%%%%%%%%%%%%%%%%%%%%%%%%%%%%%%%%%%%%

%%%%%%%%%%%%%%%%%%%%%%%%%%%%%%%%%%%%%%%%%%%%%%%%%%%%%%%%%%%%%%%%%%%%%%%%%%%%%
\def\cprime{$'$}
\providecommand{\bysame}{\leavevmode\hbox to3em{\hrulefill}\thinspace}
\providecommand{\MR}{\relax\ifhmode\unskip\space\fi MR }
\providecommand{\MRhref}[2]{%
  \href{http://www.ams.org/mathscinet-getitem?mr=#1}{#2}
}
\providecommand{\href}[2]{#2}

%%%%%%%%%%%%%%%%%%%%%%%%%%%%%%%%%%%%%%%%%%%%%%%%%%%%%%%%%%%%%%%%%%%%%%%%%%%%%

\end{document}